\documentclass[11pt]{article}

\usepackage[letterpaper, margin=1.7cm]{geometry}

\usepackage{amsmath, amsfonts, amssymb, amsthm}
\usepackage{setspace}
\usepackage{commath}
\usepackage{mathtools}
\usepackage{bbm}
\usepackage{bm}
\usepackage{subfig} 
\usepackage{color}
\usepackage[normalem]{ulem}
\usepackage{kbordermatrix}
\usepackage{mathrsfs}
\usepackage{overpic}
\usepackage{cleveref}
\usepackage{enumerate}
\usepackage[singlelinecheck=false]{caption}

\definecolor{DarkGreen}{rgb}{0.2,0.6,0.2}

\numberwithin{equation}{section}

\def\ignore#1{}

\def\ignore#1{}

\def\bZ{\mathbb Z}
\def\bN{\mathbb N}

\def\bT{\mathbb T}

\def\bP{{\mathbb P}}
\def\bE{{\mathbb E}}

\definecolor{DarkGreen}{rgb}{0.2,0.6,0.2}

\def\Ind#1{{\mathbbmss 1}_{_{\scriptstyle #1}}}

\def\eps{\varepsilon}
\def\<{\langle}
\def\>{\rangle}

 \def\ua{\uparrow}
 \def\da{\downarrow}

  \def\bbar{\overline}

\def\sgn{\text{\rm sgn}}

\def\<{\langle}\def\>{\rangle}
\def\Ind#1{{\mathbbmss 1}_{_{\scriptstyle #1}}}

\newtheorem{theorem}{Theorem}[section]

\newtheorem{lemma}[theorem]{Lemma}

\theoremstyle{definition}

\newtheorem{remark}[theorem]{Remark}

\begin{document}
\title{A probabilistic approach to the $\Phi$-variation of classical fractal functions with critical roughness}
\author{ Xiyue Han\thanks{Department of Statistics and Actuarial Science, University of Waterloo. E-mail: {\tt xiyue.han@uwaterloo.ca}} \and Alexander Schied\setcounter{footnote}{6}\thanks{Department of Statistics and Actuarial Science, University of Waterloo. E-mail: {\tt aschied@uwaterloo.ca}}
	 \and
	Zhenyuan Zhang\setcounter{footnote}{3}\thanks{
		Department of Statistics and Actuarial Science, University of Waterloo. E-mail: {\tt zhenyuan.zhang@uwaterloo.ca}
	 		\hfill\break The authors gratefully acknowledge financial support  from the
 Natural Sciences and Engineering Research Council of Canada through grant RGPIN-2017-04054}}
        
        \date{\normalsize  First version: July 2, 2020\\
   This version: September 11, 2020       }
\vspace{-2cm}
\maketitle

\vspace{-1cm}

\begin{abstract}We consider Weierstra\ss\ and Takagi-van der Waerden functions with critical degree of roughness. In this case, the functions have vanishing $p^{\text{th}}$ variation for all $p>1$ but are also nowhere differentiable and hence not of bounded variation either. We resolve this apparent puzzle by showing that 
 these functions have finite, nonzero, and linear Wiener--Young $\Phi$-variation along the sequence of $b$-adic partitions, where $\Phi(x)=x/\sqrt{-\log x}$. For the Weierstra\ss\ functions, our proof is based on the martingale central limit theorem (CLT). For the Takagi--van der Waerden functions, we use the  CLT for Markov chains if a certain parameter $b$ is odd, and the standard CLT for $b$ even. \end{abstract}
 

%
%\medskip
%
%\noindent{\it MSC 2010:} 60H05, 28A80,  26A45,  60E05

\section{Introduction and statement of results}

We consider a base function $\varphi:\mathbb R\to\mathbb R$ that is periodic with period 1 and Lipschitz continuous. 
Our aim is to study the function
\begin{align}\label{vdw}
f(t):=\sum_{m=0}^\infty \alpha^m\varphi(b^mt),\qquad t\in[0,1],
\end{align}
where  $b\in\{2,3,\dots\}$ and $\alpha\in (-1,1)$. Then the series on the right-hand side converges absolutely and uniformly in $t\in[0,1]$, so that $f$ is indeed a well defined continuous function. If 
$\varphi(t)=\nu\sin(2\pi t)+\rho\cos(2\pi t)$ for real constants $\nu$ and $\rho$, then  $f$ is a Weierstra\ss\ function. If $\varphi(t)=\min_{z\in\bZ}|z-t|$ is the tent map, then $f$ is a Takagi--van der Waerden function. It was shown in \cite{SchiedZZhang} that, under some mild conditions on $\varphi$, the function $f$ is of bounded variation for $|\alpha|<1/b$, whereas for $|\alpha|>1/b$ and $p:=-\log_{|\alpha|}b$ it has nontrivial and linear $p^{\text{th}}$ variation along the sequence 
\begin{equation}\label{b-adic partitions}
\mathbb{T}_n:=\{kb^{-n}:k=0,\dots,b^n\},\quad n\in\mathbb{N},
\end{equation}
  of $b$-adic partitions of $[0,1]$.  That is, for 
 all $t\in(0,1]$,
\begin{align}\label{pth}
\<f\>^{(q)}_t:=\lim_{n\uparrow\infty}\sum_{k=0}^{\lfloor tb^n\rfloor}\big|f((k+1)b^{-n})-f(kb^{-n})\big|^q=\begin{cases}
0&\ \text{ if } q>p,\\
t\cdot \mathbb{E}[|Z|^{q}]&\ \text{ if } q=p,\\
+\infty &\ \text{ if } q<p.\end{cases}
\end{align}
Here, $Z$ is a certain random variable, whose law is known in some special cases. For instance, if $\varphi$ is the tent map and $b$ is even, then the law of $\alpha bZ$ is the infinite Bernoulli convolution with parameter $1/(|\alpha |b)$ (see also \cite{Gantert,SchiedTakagi,MishuraSchied2} for earlier results in this special setup). Clearly, the parameter $p=-\log_{|\alpha|}b$ can be regarded as a measure for the \lq\lq roughness" of the function $f$. As a matter of fact, it is well known that a typical sample path  $t\mapsto B_H(t)$ of a fractional Brownian motion has linear $p^{\text{th}}$ variation $\<B_H\>^{(p)}_t=t\cdot \mathbb E[|B^H(1)|^p]$ for $p=1/H$.

\begin{remark}[On the connection with pathwise It\^o calculus]
Our interest in the $p^{\text{th}}$ variation of fractal functions is motivated by its connection to pathwise It\^o calculus. For instance, if  $|\alpha|=1/\sqrt b$, we have 
 $p=2$ and the limit in \eqref{pth} is just the usual quadratic variation of the function $f$, taken along the partition sequence $\{\bT_n\}_{n\in\bN}$.  It was observed by F\"ollmer \cite{FoellmerIto} that the existence of this limit is sufficient for the validity of It\^o's formula with integrator $f$, and this is the key to a rich theory of pathwise It\^o calculus with applications to robust finance; see, e.g., \cite{FoellmerSchiedBernoulli} for a discussion. Recently, Cont and Perkowski \cite{ContPerkowski} extended F\"ollmer's It\^o formula to functions with finite $p^{\text{th}}$ variation, which has led to a substantial increase in the interest in corresponding \lq\lq rough" trajectories with $p>2$.
 \end{remark}

In this note, we study the case of critical roughness, 
$\alpha=-1/b$ or $\alpha=1/b$,
in which $p=1$. For this case, it was shown in \cite{SchiedZZhang} that $\<f\>^{(q)}_t=0$ for all $q>1$ and $t\in[0,1]$. This, however, does \emph{not} imply that $f$ is of bounded variation.  For instance, if $\varphi$ is the tent map, $b=2$, and $\alpha=1/2$, then $f$ is the classical Takagi function, which is nowhere differentiable and hence cannot be of bounded variation; a very short proof of this fact was given by de Rham \cite{deRham} and later rediscovered by    Billingsley \cite{Billingsley1982}. For the Weierstra\ss\ function, the proof of nowhere differentiability for all $\alpha\in[1/b,1)$ is more difficult. Starting from Weierstra\ss's original work,
% \cite{Weierstrass}, 
it attracted numerous authors until a definite result was given by  Hardy \cite{Hardy1916}.

It is therefore apparent that, in the critical case $|\alpha|=1/b$, power variation $\<f\>^{(q)}$ is insufficient to capture the exact degree of roughness of the function $f$. To give a precise result on the roughness of the function $f$ in the critical case, we take a strictly increasing function $\Phi:[0,1)\to[0,\infty)$ and investigate the limit
$$
\<f\>^\Phi_t:=\lim_{n\uparrow\infty}\sum_{k=0}^{\lfloor tb^n\rfloor}\Phi\big(|f((k+1)b^{-n})-f(kb^{-n})|\big),
$$
which can be regarded as the Wiener--Young $\Phi$-variation of $f$ (see, e.g., \cite{Appelletal}), restricted to the sequence of $b$-adic partitions \eqref{b-adic partitions}. Our main results will show that the correct choice for $\Phi$ is the function
$$\Phi(x)=\frac{x}{\sqrt{-\log x}}\quad\text{for $x\in(0,1)$\quad and}\quad \Phi(0):=0.
$$
We fix this function $\Phi$ throughout the remainder of this paper.
Our first result establishes the $\Phi$-variation of $f$ from \eqref{vdw} for the class of Takagi--van der Waerden functions. 

\begin{theorem}\label{TvdW thm}Let $\varphi(t)=\min_{z\in\bZ}|z-t|$ be the tent map, $b\in\{2,3,\dots\}$, and $|\alpha|=1/b$. Then  the $\Phi$-variation of the Takagi--van der Waerden function $f$ exists along $\{\bT_n\}_{n\in\bN}$.
 If $b$ is even,  then it is given by
$$\<f\>^\Phi_t=t\cdot\sqrt{\frac2{\pi\log b}},\qquad t\in[0,1].
$$
 If $b$ is odd, then 
$$\<f\>^\Phi_t=t\cdot\sqrt{\frac{2(b+\sgn(\alpha))}{\pi(b-\sgn(\alpha))\log b}},\qquad t\in[0,1].
$$
\end{theorem}

Our results will be consequences of suitable central limit theorems (CLTs). In the preceding theorem, the case of $b$ even  will be settled by the standard CLT, whereas the case of $b$ odd will require the use of a CLT for Markov chains. For establishing the $\Phi$-variation of the critical Weierstra\ss\ functions, as stated in the following theorem, we rely on the martingale CLT. A loosely related CLT for the classical Takagi function was proved by Gamkrelidze \cite{Gamkrelidze}.

\begin{theorem}\label{Weierstrass thm}Suppose $\varphi(t)=\nu\sin(2\pi t)+\rho\cos(2\pi t)$, $b\in\{2,3,\dots\}$, and $|\alpha|=1/b$. Then  the $\Phi$-variation of the Weierstra\ss\ function $f$ exists along $\{\bT_n\}_{n\in\bN}$ and is given by
$$\<f\>^\Phi_t=t\cdot2\sqrt{\frac{\pi(\nu^2+\rho^2)}{\log b}},\qquad t\in[0,1].
$$
\end{theorem}

%
%\begin{remark} 
%Gamkrelidze \cite{Gamkrelidze} proved the following, loosely related CLT for the classical Takagi function ($\varphi=$ tent map, $b=2$): The Lebesgue measure of the set $$\bigg\{x\in[0,1]:\frac{f(x+h)-f(x)}{h\sqrt{\log_2h}}<y\bigg\}
%$$
%converges to $(2\pi)^{-1/2}\int_{-\infty}^ye^{-z^2/2}\,dz$ as $h\da0$. In \cite{deLimaSmania}, this result was recently extended to a general class of fractal functions.
%\end{remark}

\section{Proofs}\label{proofs section}

We first consider only the $\Phi$-variation $\<f\>^\Phi_t$ for $t=1$. The case  $t<1$ will be discussed  at the end of this section, simultaneously for both theorems. 
 We
 fix $b\in\{2,3,\dots\}$ and $\alpha\in\{-1/b,+1/b\}$. 
Following \cite{SchiedZZhang}, we let $(\Omega,\mathscr{F},\mathbb{P})$ be a probability space supporting an independent sequence $U_1,U_2,\dots$  of random variables with 
a uniform distribution on $\{0,1,\dots, b-1\}$ and define the stochastic process
$
R_m:=\sum_{i=1}^mU_ib^{i-1}.
$
Note that $R_m$ has a uniform distribution on $\{0,\dots, b^m-1\}$. 
Therefore, for $n\in\bN$ such that all increments $\big|f((k+1)b^{-n})-f(kb^{-n})\big|$ are less than 1,
\begin{equation}\label{Vn def eq}
V_n:=\sum_{k=0}^{b^n-1}\Phi\big(\big|f((k+1)b^{-n})-f(kb^{-n})\big|\big)=b^n\mathbb E\Big[\Phi\big(\big|f((R_n+1)b^{-n})-f(R_nb^{-n})\big|\big)\Big].
\end{equation}
To simplify the expectation on the right, let the $n^{\text{th}}$ truncation of $f$ be given by
$
f_n(t)=\sum_{m=0}^{n-1} \alpha^m\varphi(b^mt).
$
The periodicity of $\varphi$ implies that
\begin{align*}
f((R_n+1)b^{-n})-f(R_nb^{-n})&=f_n((R_n+1)b^{-n})-f_n(R_nb^{-n})\\
%&=\sum_{m=0}^{n-1}\alpha^mb^{m-n}\frac{\varphi((R_n+1)b^{m-n})-\varphi(R_nb^{m-n})}{b^{m-n}}\\
&=b^{-n} \sgn(\alpha)^n\sum_{m=1}^{n}\sgn(\alpha)^{m}\frac{\varphi((R_n+1)b^{-m})-\varphi(R_nb^{-m})}{b^{-m}}.
\end{align*}
The periodicity of $\varphi$ implies moreover that for $m\le n$,
$$\varphi\big(x+R_nb^{-m}\big)=\varphi\Big(x+\sum_{i=1}^nU_ib^{i-1-m}\Big)=\varphi\Big(x+\sum_{i=1}^mU_ib^{i-1-m}\Big)=\varphi\big(x+R_mb^{-m}\big).
$$
Therefore,
$$\sgn(\alpha)^{m}\frac{\varphi((R_n+1)b^{-m})-\varphi(R_nb^{-m})}{b^{-m}}=\sgn(\alpha)^{m}\frac{\varphi((R_m+1)b^{-m})-\varphi(R_mb^{-m})}{b^{-m}}=:Y_m.
$$
It follows that 
\begin{align}\label{Vn eq}
V_n&=b^n\mathbb{E}\bigg[\Phi\Big(b^{-n}\Big|\sum_{m=1}^{n}Y_m\Big|\Big)\bigg].\
\end{align}

\begin{lemma}\label{CLT lemma}Suppose that $Z_0,Z_1,Z_2,\dots$ is a sequence of random variables with $Z_0=0$ and uniformly bounded increments such that the laws of $\frac1{\sqrt n}Z_n$ converge weakly to some  normal distribution $N(0,\sigma^2)$ with $\sigma^2>0$ and that the expression $\frac1n\bE[Z_n^2]$ is bounded in $n$. Then
$$b^n\bE\Big[\Phi\big(b^{-n}\big|Z_n\big|\big)\Big]\longrightarrow \sqrt{\frac{2\sigma^2}{\pi\log b}}.
$$
\end{lemma}

\begin{proof}The fact that $\frac1n\bE[Z_n^2]$ is bounded implies together with  standard arguments that for every nondegenerate interval $I\subset[0,\infty)$,
\begin{equation}\label{moment CLT eq}
\lim_{n\uparrow\infty}\mathbb E\Big[\Ind{\{|\frac1{\sqrt n}Z_n|\in I\}}\Big|\frac1{\sqrt n}Z_n\Big|\Big]= \frac1{\sqrt{2\pi\sigma^2}}\int_{\{|z|\in I\}}|z|e^{-z^2/(2\sigma^2)}\,dz.\end{equation}
We have
$$b^n\bE\Big[\Phi\big(b^{-n}\big|Z_n\big|\big)\Big]=\mathbb E\bigg[\frac{|Z_n|}{\sqrt{n\log b-\log |Z_n|}}\Ind{\{|Z_n|>0\}}\bigg].$$
Let $C$ be an almost sure uniform bound for $|Z_{k+1}-Z_k|$. Hence, for all $\beta\in(0,\log b)$ there exists $n_0\in\bN$ such that $n\beta<n\log b-\log (Cn)$ for all $n\ge n_0$. Hence, 
\begin{equation}\label{log lower estimate}
\sqrt{n\log b-\log  |Z_n|}\ge \sqrt{n\beta}\quad\text{for $n\ge n_0$,}
\end{equation}
and taking $I:=(0,\infty)$ in \eqref{moment CLT eq} gives
$$\limsup_{n\ua\infty}b^n\bE\Big[\Phi\big(b^{-n}\big|Z_n\big|\big)\Big]\le \frac1{\sqrt{2\pi\sigma^2\beta}}\int_{\{|z|\in I\}}|z|e^{-z^2/(2\sigma^2)}\,dz=\sqrt{\frac{2\sigma^2}{\pi\beta}}.
$$
To get a lower bound, observe that for every $\eps>0$ and $n\ge 1/\eps^2$, 
$$\Ind{\{|\frac1{\sqrt n}Z_n|\ge\eps\}}\sqrt{n\log b-\log |Z_n|}\le \Ind{\{|\frac1{\sqrt n}|Z_n|\ge\eps\}}\sqrt{n\log b}.
$$
Hence, we get from \eqref{moment CLT eq} that 
$$\liminf_{n\ua\infty} b^n\bE\Big[\Phi\big(b^{-n}\big|Z_n\big|\big)\Big]\ge \frac1{\sqrt{2\pi\sigma^2\log b}}\int_{\{|z|\ge\eps\}}|z|e^{-z^2/(2\sigma^2)}\,dz.$$
Sending $\eps\downarrow 0$ and $\beta\ua\log b$ gives the result.
\end{proof}

\begin{proof}[Proof of Theorem \ref{TvdW thm} for $t=1$.]  For  $b$ even, \cite[Proposition 3.2 (a)]{SchiedZZhang} states that $Y_1,Y_2,\dots$ is an i.i.d.~sequence of symmetric $\{-1,+1\}$-valued Bernoulli random variables. Therefore,  \eqref{Vn eq}, the classical CLT, and Lemma \ref{CLT lemma} give
$V_n\to\sqrt{2/({\pi\log b})}$.
If $ b$ is odd, then \cite[Proposition 3.2 (b)]{SchiedZZhang} states that the random variables $\sgn(\alpha)^mY_m$ form a time-homogeneous Markov chain on $\{-1,0,+1\}$ with initial distribution $\mu_1=(\frac{b-1}{2b},\frac{1}{b},\frac{b-1}{2b})$  and transition matrix $P_+$, where
\[
  P_\pm:=\frac1{2b} \begin{pmatrix} b\pm1 & 0 & b\mp1 \\
   b-1 &2&b-1  \\
    b\mp1& 0 & b\pm1 
\end{pmatrix}.\]
It follows that $Y_1,Y_2,\dots$ also form a time-homogeneous Markov chain with  initial distribution $\mu_1$  and transition matrix $P_+$ for $\alpha>0$ and  $P_-$ for $\alpha<0$. 
 Since $0$ is a transient state, we can clearly consider only the restriction of the Markov chain to its positive recurrent states, $-1$ and $+1$. Let $\bbar P_\pm$ be the $2\times2$-matrix obtained from $P_\pm$ by deleting the second row and second column from $P$, and define $\bar\mu_1=(1/2,1/2)$. Then $\bar\mu_1$ is the unique stationary distribution for $\bbar P_\pm$.  Moreover,
 $$\bbar P^n_\pm=\frac12\begin{pmatrix}1+(\pm b)^{-n}&1-(\pm b)^{-n}\\
 1-(\pm b)^{-n}&1+(\pm b)^{-n}
 \end{pmatrix}.
 $$
For the state-constraint Markov chain $\bbar Y_1,\bbar Y_2,\dots$ with initial distribution $\bar\mu_1$ and transition matrix $\bbar P_\pm$, we thus have 
$\text{var}(\bbar Y_1)=1$ and 
\begin{align*}
\text{cov}(\bbar Y_1,\bbar Y_{n+1})&=\sum_{y_1,y_{n+1}\in\{-1,+1\}}\bar\mu_1(y_1)\bbar P_\pm^n(y_1,y_{n+1})y_1y_{n+1}
%&=\frac14\Big((1+(\pm b)^{-n})(-1)^2+2(1-(\pm b)^{-n})(-1)(+1)+(1+(\pm b)^{-n})(+1)^2\Big)\\
%&
=(\pm b)^{-n}.
\end{align*}
Letting
$$\sigma^2:=\text{var}(\bbar Y_1)+2\sum_{n=1}^\infty \text{cov}(\bbar Y_1,\bbar Y_{n+1})=\frac{b\pm 1}{b\mp 1},
$$
the central limit theorem for Markov chains (see, e.g., \cite{JonesMarkovCLT}) implies that $\frac1{\sqrt{n}}\sum_{k=1}^n\bbar Y_k$ converges in law to $N(0,\sigma^2)$. Due to the stationarity of the Markov chain, we have moreover 
\begin{align*}
\mathbb E\Big[\Big(\frac1{\sqrt n}\sum_{k=1}^n\bbar Y_k\Big)^2\Big]&=\frac1n \sum_{k=1}^n\text{var}(\bbar Y_k)+\frac2n \sum_{k=1}^{n-1}\sum_{\ell=k+1}^n \text{cov}(\bbar Y_k,\bbar Y_{\ell})\\
&=1+\frac2n \sum_{k=1}^{n-1}\sum_{\ell=k+1}^n (\pm b)^{k-\ell}\le 1+\frac2n\cdot \frac{ b^{1-n}+ b n+ b-n}{(b- 1)^2},
\end{align*}
which is uniformly bounded in $n$. Therefore, Lemma \ref{CLT lemma} and \eqref{Vn eq} give
$V_n\to\sqrt{2(b\pm 1)/(\pi(b\mp 1)\log b)}$.
\end{proof}

Now we prepare for the proof of Theorem \ref{Weierstrass thm} for $t=1$. Let $\mathscr F_0=\{\emptyset,\Omega\}$ and $\mathscr F_n:=\sigma(U_1,\dots, U_n)$ for $n\in\mathbb N$. Then each $Y_n$ is $\mathscr F_n$-measurable.
Since $U_1,\dots, U_{n}$ can be recovered from $R_{n}$, we have $\mathscr F_n=\sigma(R_n)$ for $n\ge1$. We define 
$Z_0:=0$ and $ Z_n:=\sum_{k=1}^nY_k$ for $n\in\mathbb N$.

\begin{lemma}If $\varphi(t)=\nu\sin(2\pi t)+\rho\cos(2\pi t)$, then $\{Z_n\}_{n\in\mathbb N_0}$ is a martingale with respect to $\{\mathscr F_n\}_{n\in\mathbb N_0}$.
\end{lemma}

\begin{proof}We must show that $\mathbb E[Y_n|R_{n-1}]=0$ $\mathbb P$-a.s.~for $n\ge1$. To this end, we use that $R_n=R_{n-1}+U_nb^{n-1}$, where $R_0:=0$ and $U_n$ is independent of $R_{n-1}$. Therefore,
\begin{align}
\mathbb E[Y_n|R_{n-1}=r]&=(\sgn(\alpha))^n\mathbb E\bigg[\frac{\varphi\big((r+U_nb^{n-1}+1)b^{-n}\big)-\varphi\big((r+U_nb^{n-1})b^{-n}\big)}{b^{-n}}\bigg]\nonumber\\
&=\frac{(\sgn(\alpha)b)^n}b\sum_{k=0}^{b-1}\Big( \varphi\big((r+1)b^{-n}+k/b\big)-\varphi\big(rb^{-n}+k/b\big)\Big).\label{cond exp sum eq}
\end{align}
If $n=1$, then $r$ must be zero, and the sum in \eqref{cond exp sum eq}
 is a telescopic sum with value $\varphi(1)-\varphi(0)=0$. Now consider the case $n\ge2$. Then, for all $x\in\mathbb R$, $i=\sqrt{-1}$, and $\mathfrak {Re}$ denoting the real part of a complex number,
 \begin{align*}
 \sum_{k=0}^{b-1}\varphi(x+k/b)&=\mathfrak {Re}\bigg((\rho-i\nu)\sum_{k=0}^{b-1}e^{2\pi i(x+k/b)}\bigg)
% =\mathfrak {Re}\bigg((\rho-i\nu)e^{2\pi ix}\sum_{k=0}^{b-1}(e^{2\pi i/b})^k\bigg)\\
% &
 =\mathfrak {Re}\bigg((\rho-i\nu)e^{2\pi ix}\cdot\frac{e^{2\pi ib/b}-1}{e^{2\pi i/b}-1}\bigg)=0.
 \end{align*}
 Therefore, the sum in \eqref{cond exp sum eq}
 vanishes. 
\end{proof}

\begin{lemma}\label{equidist lemma}With $\delta_x$ denoting the Dirac measure in $x\in\mathbb R$ and $\lambda$ denoting the Lebesgue measure on $[0,1]$, we have $\mathbb P$-a.s.,
$\frac1n\sum_{k=1}^n\delta_{b^{-k}R_k}\to\lambda$ weakly as $n\uparrow\infty$.\end{lemma}

\begin{proof} Without loss of generality, we can extend the sequence $\{U_i\}_{i\in\mathbb N}$ to a two-sided sequence $\{U_i\}_{i\in\mathbb Z}$ of i.i.d.~random variables with a uniform distribution on $\{0,\dots, b-1\}$. Then we define
$X_n:=\sum_{j=1}^\infty U_{n+1-j}b^{-j}=\sum_{j=0}^\infty U_{n-j}b^{-(j+1)}$ for $n\in\mathbb Z$.
Each $X_n$ is uniformly distributed on $[0,1]$, i.e., has law $\lambda$. Moreover, in comparison with $X_n$, the random variable $X_{n+1}$ is obtained by shifting the sequence  $\{U_i\}_{i\in\mathbb Z}$ one step to the right. It is well-known that the dynamical system corresponding to such a two-sided Bernoulli shift is mixing and hence ergodic (for a proof, see, e.g., Example 20.26 in \cite{Klenke}). By Birkhoff's ergodic theorem, we thus have 
$\frac1n\sum_{k=1}^nf(X_k)\to \int_0^1f\,d\lambda
$
$\mathbb P$-a.s.~for each bounded Borel-measurable function on $[0,1]$. Since $|b^{-n}R_n-X_n|\le b^{-n}$, we hence obtain $\frac1n\sum_{k=1}^nf(b^{-k}R_k)\to \int_0^1f\,d\lambda
$
$\mathbb P$-a.s.~for each (uniformly) continuous function on $[0,1]$. Since $C[0,1]$ is separable, the result follows.
\end{proof}

\noindent{\it Proof of Theorem \ref{Weierstrass thm} for $t=1$.}  
Let
$\<Z\>_n:=\sum_{k=1}^n\mathbb E[Y_k^2|\mathscr F_{k-1}]$
be the predictable quadratic variation of the martingale $\{Z_n\}_{n\in\mathbb N_0}$. We define
$\psi_n(x):=(\varphi(x+b^{-n})-\varphi(x))/b^{-n}
$.
Then $\psi_n(x)\to\varphi'(x)$ uniformly in $x$. By arguing as in \eqref{cond exp sum eq}, we see that 
$\mathbb E[Y_k^2|\mathscr F_{k-1}]=\frac1b\sum_{\ell=0}^{b-1}\big(\psi_k(b^{-k}R_{k-1}+\ell/b)\big)^2.
$
We therefore conclude from Lemma \ref{equidist lemma} that 
$$\frac1n\<Z\>_n=\frac1b\sum_{\ell=0}^{b-1}\sum_{k=1}^n\big(\psi_k(b^{-k}R_{k-1}+\ell/b)\big)^2\longrightarrow \int_0^1\big(\varphi'(t)\big)^2\,dt=2\pi^2(\nu^2+\rho^2)=:\sigma^2.
$$
Analogously, one sees easily that 
 $s_n^2:=\mathbb E[\<Z\>_n]$ satisfies $\frac1n s_n^2\to\sigma^2$. 
Since the increments $Y_k$ are uniformly bounded, the Lindeberg condition,
$$\frac1{n}\sum_{k=1}^n\mathbb E\big[Y_k^2\Ind{\{Y_k^2\ge \varepsilon n\}}\big|\mathscr F_{k-1}\big]\longrightarrow 0\qquad\text{$\bP$-a.s.~for all $\eps>0$,}
$$
is clearly satisfied. Therefore, the martingale central limit theorem in the form of \cite[(7.4) in Chapter 7]{Durrett} yields that the laws of $\frac1{\sqrt{n}}Z_n$ converge weakly to $N(0,\sigma^2)$. Lemma \ref{CLT lemma} hence gives
$$V_n\longrightarrow 2\sqrt{\frac{\pi(\nu^2+\rho^2)}{\log b}}.\eqno{\qedsymbol}
$$
\bigskip

Finally, we show how the preceding results can be extended to the case $0\le t<1$. Writing $Z_n$ for $\sum_{k=1}^nY_k$, the $\Phi$-variation over the interval $[0,t]$ is equal to
\begin{align*}
V_{n,t}&:=\sum_{k=0}^{b^n-1}\Phi\big(\big|f((k+1)b^{-n})-f(kb^{-n})\big|\big)\Ind{[0,t]}(kb^{-n})=b^n\mathbb E\Big[\Phi\big(\big|f((R_n+1)b^{-n})-f(R_nb^{-n})\big|\big)\Ind{\{b^{-n}R_n\le t\}}\Big]\\
&=b^n\mathbb{E}\bigg[\Phi\Big(b^{-n}\Big|\sum_{m=1}^{n}Y_m\Big|\Big)\Ind{\{b^{-n}R_n\le t\}}\bigg]=\mathbb E\bigg[\frac{|Z_n|}{\sqrt{n\log b-\log |Z_n|}}\Ind{\{|Z_n|>0\}}\Ind{\{b^{-n}R_n\le t\}}\bigg].
\end{align*}
Let $\delta>0$ be given and pick $m\in\bN$ such that $b^{-m}\le\delta$. Clearly, $\{b^{-n}R_n\le t\}\subset\{b^{-n}R_{m,n}\le t\}$, where $R_{m,n}:=R_n-R_{n-m}=\sum_{k=n-m+1}^nU_kb^{k-1}$.  In addition, we argue as in the proof of Lemma \ref{CLT lemma} and take $\beta\in(0,\log b)$ and $n_0\in\bN$ such that $n\beta<n\log b-\log (Cn)$ for all $n\ge n_0$ and \eqref{log lower estimate} holds. Therefore, for $n\ge m\vee n_0$.
\begin{align*}
V_{n,t}&\le \frac1{\sqrt{n\beta}}\bE\big[|Z_n|\Ind{\{b^{-n}R_{m,n}\le t\}}\big]\le  \frac1{\sqrt{n\beta}}\bE\big[|Z_{n-m}|\Ind{\{b^{-n}R_{m,n}\le t\}}\big]+\frac1{\sqrt{n\beta}}\bE\Big[\Big|\sum_{k=n-m+1}^nY_k\Big|\Big].
\end{align*}
Clearly, the rightmost term converges to zero as $n\ua\infty$. Moreover, $Z_{n-m}$ and $R_{m,n}$ are independent, and so
$$\limsup_{n\ua\infty}V_{n,t}\le \sqrt{\frac{2\sigma^2}{\pi\beta}}\limsup_{n\ua\infty}\bP[b^{-n}R_{m,n}\le t]
\le \sqrt{\frac{2\sigma^2}{\pi\beta}}\limsup_{n\ua\infty}\bP[b^{-n}R_{n}\le t+\delta]=\sqrt{\frac{2\sigma^2}{\pi\beta}}(t+\delta),
$$
where the second inequality follows from the fact that $b^{-n}R_{m,n}\ge b^{-n}R_n-\delta$ for $n>m$. Sending $\beta\ua\log b$ and $\delta\da0$ gives the desired upper bound.

To get a corresponding lower bound, we choose $\delta>0$ and $m$ as in the upper bound. In addition, we choose  $\eps>0$.  For $n> m\vee1/\eps^2$, 
we then get as in the proof of Lemma \ref{CLT lemma},
\begin{align*}
V_{n,t}&\ge \mathbb E\bigg[\frac{|Z_n|}{\sqrt{n\log b}}\Ind{\{|\frac1{\sqrt n}|Z_n|\ge\eps\}}\Ind{\{b^{-n}R_n\le t\}}\bigg] \ge  \mathbb E\bigg[\frac{|Z_n|}{\sqrt{n\log b}}\Ind{\{|\frac1{\sqrt n}|Z_n|\ge\eps\}}\Ind{\{b^{-n}R_{m,n}\le t-\delta\}}\bigg].
\end{align*}
Now let $C$ be a uniform upper bound for $|Y_k|$ and choose $n_1$ such that $mC\le\eps\sqrt{n_1}$. Then, for $n\ge n_1\vee m\vee 1/\eps^2$,
\begin{align*}
V_{n,t}&\ge \mathbb E\bigg[\frac{|Z_{n-m}|}{\sqrt{n\log b}}\Ind{\{|\frac1{\sqrt n}|Z_{n-m}|\ge2\eps\}}\Ind{\{b^{-n}R_{m,n}\le t-\delta\}}\bigg]-\frac1{\sqrt{n\log b}}\bE\Big[\Big|\sum_{k=n-m+1}^nY_k\Big|\Big].
\end{align*}
Again, the second expectation on the right converges to zero. Using as before the independence of $Z_{n-m}$ and $R_{m,n}$ now easily gives the desired lower bound. This concludes the proof of Theorems \ref{TvdW thm} and \ref{Weierstrass thm} for $0\le t<1$.

\parskip-0.5em\renewcommand{\baselinestretch}{0.9}\normalsize
\bibliography{CTbook}{}
\bibliographystyle{plain}
\end{document}